\documentclass[12pt,a4paper]{amsart}
\usepackage{amsmath,amssymb,epsf,graphics}
\usepackage{amsfonts,amsxtra,amscd,verbatim,eucal}
\usepackage{enumerate}
\usepackage[all]{xy}

\newtheorem{thm}{Theorem}[section]
\newtheorem{lemma}{Lemma}
\newtheorem{proposition}{Proposition}

\newtheorem{defn}{Definition}
\newtheorem{question}{Question}

\theoremstyle{remark}
\newtheorem{rmk}{Remark}

\begin{document}
\author{Catriona Maclean\\
Institut Fourier\\ Universit\'e Grenoble Alpes.}
\title{Approximable algebras and a question of H. Chen.}
\maketitle
\begin{abstract}
In \cite{volumes_chen}, Huayi Chen introduces the notion of an approximable 
graded algebra, and asks if any such algebra is a sub-algebra of the graded
section ring of a big line bundle on an algebraic variety. We give a 
counter-example showing that this is not the case.
\end{abstract}
\section{Introduction}
The Fujita approximation theorem, \cite{fuj}, is an important result in 
algebraic geometry. It states that whilst the section ring associated to a 
big line bundle $L$ on an algebraic variety $X$ 
\[ R(L)\stackrel{\rm def}{=} \oplus_n H^0(nL, X)\]
is typically not a finitely generated algebra, it can be approximated
arbitrarily well by finitely generated algebras. More precisely,  we have that
\begin{thm}[Fujita]
Let $X$ be an algebraic variety and let $L$ be a big line bundle on $X$. 
For any $\epsilon>0$ there exists a birational modification 
\[\pi: \hat{X}\rightarrow X\]
and a decomposition of $\mathbb{Q}$ divisors, $\pi^* (L)= A+E$ such that
\begin{itemize}
\item $A$ is ample and $E$ is effective,
\item ${\rm vol}(A)\geq (1-\epsilon){\rm vol}(L)$.
\end{itemize}
\end{thm}
In \cite{LM}, Lazarsfeld and Mustata used the Newton-Okounkov body associated to
$A$ to give a simple proof of Fujita approximation. The Newton-Okounkov body, 
constructed in \cite{KK} and \cite{LM}, building on previous work of Okounkov
\cite{Okounkov}, is a convex body $\Delta_{Y_\bullet}(L,X)$ in 
$\mathbb{R}^d$ associated to the 
data of
\begin{itemize}
\item a $d$-dimensional variety $X$
\item an admissible flag $Y_\bullet$ on $X$
\item a big line bundle $L$ on $X$.
\end{itemize}
This convex body encodes information
on the asymptotic behaviour of the spaces of global sections $H^0(nL)$ for large
values of $L$. \\ \\
Lazarsfeld and 
Mustata's simple proof of Fujita approximation is based on the equality of 
volumes of Newton-Okounkov bodies
\begin{equation}\label{volumes} {\rm vol}(L)= d! {\rm vol}
(\Delta_{Y_\bullet}(L,X))\end{equation}
where we recall that the volume of a big line bundle on a $d$-dimesional variety
is defined by
\[ {\rm vol}(L)= \lim_{n\rightarrow \infty}\frac{ d! h^0(nL)}{n^d}.\]
One advantage of their approach to the Fujita theorem
is that Newton-Okounkov bodies are not 
only defined for section algebras $R(L)$, but also for any graded sub-algebra of
section algebras. Lazarsfeld and Mustata give combinatorical 
conditions (conditions 2.3-2.5 of \cite{LM}) under which equation
\ref{volumes} holds for a graded sub-algebra $\oplus_n B_n\subset R(L)$
and show that these conditions hold if 
the subalgebra $\oplus_n B_n$ 
{\it contains an ample series}.\footnote{Ie. 
if there exists an ample divisor $A\leq L$ such that $\oplus_n H^0(\lfloor
nA \rfloor)\subset B$}\\ \\
Di Biagio and Pacenzia in \cite{dBP}
subsequently used Newton-Okounkov bodies associated to 
restricted algebras to prove a Fujita approximation theorem for restricted 
linear series, ie. subalgebras of $\oplus_n H^0(nL|_V,V)$ obtained as the
restriction of the complete algebra $\oplus_n H^0(nL,X)$, where $V\subset X$ is a 
subvariety. \\ \\
In \cite{volumes_chen}, Huayi Chen uses Lazarsfeld and Mustata's work on 
Fujita approximation to
prove a Fujita-type approximation theorem in the arithmetic setting.
In the course of this work he defines the notion of approximable 
graded algebras, which are exactly those algebras for which
a Fujita-type approximation theorem hold. 
\begin{defn}
An integral graded algebra $B=\oplus_n B_n$ is approximable if and only if
the following conditions are satisfied.
\begin{enumerate}
\item all the graded pieces $B_n$ are finite dimensional.
\item for all sufficiently large $n$ the space $B_n$ is non-empty
\item  
for any $\epsilon$ there exists an $p_0$ such that for all $p\geq p_0$ 
we have that 
\[ \liminf_{n\rightarrow \infty} \frac{ {\rm dim}({\rm Im}(S^nB_p\rightarrow B_{np}))}{{\rm dim}(B_{np})}> (1-\epsilon). \]
\end{enumerate}
\end{defn}
In his paper \cite{volumes_chen} 
Chen asks the following question.
\begin{question}[Chen]
Let $B=\oplus_n B_n$ be an approximable graded algebra. Does there always exist 
a variety $X$ and a big line bundle $L$ such that $B$ can be included as a 
subalgebra of $R(L)$ ?
\end{question}
We will say that an approximable algebra for which the answer is ``no'' is {\it non-sectional}.\\ \\
The aim of this note is 
to prove the following theorem
\begin{thm}\label{maintheorem}
There exists an approximable graded algebra $B=\oplus_n B_n$ which is 
non-sectional.  
\end{thm}
The basic criterion used will be the 
following : if $B$ is an approximable graded algebra, we can consider the
associated field of homogeneous quotients $K^{\rm gr}(B)$, into which, after 
choice of a base element $b\in B_1$, the algebra $B$ can be included.
\\ \\
If $B$
is a graded sub-algebra of $\oplus_n H^0(nL)$ for some big line bundle $L$ then
the number of valuations of $K^{\rm gr}(B)$ which take negative values on 
certains elements  of $B$ is finite. \\ \\
In section 2) below we will state and prove our criterion for 
non-sectionality, whereas in section 3) we construct the non-sectional approximable algebra. \\ \\
Note finally that while our counter example is not associated to a line bundle,
it is of the form
\[ \oplus_n H^0(\lfloor nD \rfloor)\]
for a certain  infinite divisor $D=\sum_{i=0}^\infty  a_i D_i$ with the property that 
$\lfloor nD \rfloor$ is finite for any choice of $n$. This poses the following
natural question.
\begin{question}
Let $B=\oplus_n B_n$ be a graded approximable algebra. Does there necessarily exist a variety $X$ and a 
countable formal sum  \[D=\oplus_i a_i D_i\] of divisors on $X$ 
such that for all $n$ the integral divisor $\lfloor nD\rfloor$ is finite and $B$ is isomorphic to the
graded algebra $\oplus_n \mathcal{O}\left(\lfloor nD\rfloor\right)$ ?
\end{question}
\section{A criterion for not-sectionality.}
In this section, we will give a  criterion
for non-sectionality of a approximable graded algebra in terms of valuations on the
associated field of homogeneous degree-zero rational functions.
\begin{defn}
Let $B=\oplus_n B_n$ be an integral graded algebra. 
We define its graded field of
functions by
\[ K^{\rm gr}(B)=\left\{ \frac{f}{g}| \exists n\mbox{ s.t. } f,g\in B_n\right\}
\big{/}\sim\]
where $ \frac{f_1}{g_1}\sim \frac{f_2}{g_2} \mbox{ iff } f_1g_2=f_2g_1$.
\end{defn}
We now state our criterion.
\begin{proposition}\label{criterion}
Let $\oplus_n B_n$ be an approximable graded algebra with $B_0=\mathbb{C}$. 
Suppose that the graded field of functions $K^{\rm gr}(B)$ is a
finitely generated field extension of $\mathbb{C}$ 
with transcendence degree one and there is a 
choice of element $b_1\in B_1$ such that the set
\[ \nu_{B, b_1}=\left\{ 
\nu\mbox{ valuation } on K^{\rm gr}(B)| 
\exists n\in \mathbb{N}, b'\in B_n \mbox{  such that } 
\nu_B\left(\frac{b'}{b_1^n}\right)<0\right\}\]
is infinite. Then there does not exist a big line bundle $L$ on a variety $X$ such 
that $\oplus_n B_n$ is isomorphic as a graded algebra 
to a graded sub-algebra of $\oplus_n H^0(nL,X)$.
\end{proposition}
\begin{proof}
Suppose that on the contrary we have a big line bundle $L$ on a complex variety
$X$  and a graded inclusion $i: B\rightarrow R(L)$. This gives
rise to a field inclusion 
\[ i^{\rm gr} : K^{\rm gr}(B)\rightarrow K^{\rm gr}(L)= K(X).\]
After blowing up, we obtain a birational morphism 
$\pi^{\rm sm}: X^{\rm sm}\rightarrow 
X$ from a {\it smooth} variety $X^{\rm sm}$. Consider the pull-back 
$\pi^{\rm sm *}(L)$: there is an injective pull-back morphism
$\pi^{\rm sm *} : R(L)\rightarrow R(\pi^*(L))$. Replacing $X$ with $X^{\rm sm}$ and 
$i^{\rm gr}$ with $\pi^*\circ i^{\rm gr}$, we may assume that the variety $X$
is smooth. \\ \\
We have a distinguished element $b_1\in B_1$, and an associated element 
$\sigma_1= i^{\rm gr}(b_1)\in H^0(L)$, so there are induced injective morphisms
\[ \phi_B: B\hookrightarrow K^{\rm gr}(B)\]
\[\phi_L:  R(L)\hookrightarrow K^{\rm gr}(L)=K(X)\]
given by  $\phi_B(b_n)= \frac{b_n}{b_1^n}$ for all $b_n\in B_n$ and 
$\phi_L(\sigma_n)=\frac{\sigma_n}{(\sigma_1)^n}$ for all $\sigma_n\in  H^0        (nL)$. There is also an induced 
inclusion \[ i^{\rm gr}: K^{\rm gr}(B)\rightarrow 
K^{\rm gr}(X)\] given by $i^{\rm gr}\left(\frac{f}{g}\right)=\frac{i(f)}{i(g)}$.
By assumption,  $K^{\rm gr}(B)$ is finitely generated as a field extension of 
$\mathbb{C}$ and is of transcendence degree $1$ so
there is a unique smooth complex curve $C$ such that 
$K^{\rm gr}(B)\sim K(C)$. After fixing an isomorphism 
$I: K(C)\rightarrow K^{\rm gr}(B)$  we have an inclusion
\[ K(C)\stackrel{ i^{\rm gr}\circ I}{\hookrightarrow} K(X)\]
and an induced rational morphism with dense image
\[\pi_X:  X---> C\]
such that $i^{\rm gr}\circ I= \pi_X^*$. After birational modification, 
we may assume that the map $\pi_X$ is in fact
a surjective morphism. \\ \\
Valuations of $K(C)$ correspond to points of the curve $C$. Consider the set of 
valuations in $\nu_{B, b_1}$, which by hypothesis is infinite 
\[ \nu_{B,b_1}= \{\nu_1,\nu_2,\ldots, \nu_m,\ldots\}\] and consider the set of 
associated points in $C$, \[\{p_1,p_2,\ldots, p_m,\ldots,\}\]
having the property that $\nu_j\circ I= \nu_{p_j}$. 
For any $j$ we know that there exists a positive integer $n(j)$
and an element $b_j\in B_{n(j)}$ for some $n$ such that 
\[\nu_j\left(\frac{b_j}{b_1^{n(j)}}\right)<0.\] There exists a meromorphic function on 
$C$, $f$, such that  $\frac{b_j}{b^{n(j)}}=I(f)$.  Consider 
$i^{\rm gr}(\frac{b_n}{b^n})\in K(X)$ : we have that 
\[i^{\rm gr}\left(\frac{b_n}{b^n}\right) = i^{\rm gr}(I(f))= \pi_X^{*}(f)\]
and hence
\[ \pi_X^*(f)=
\frac{i^{gr}(b_j)}{i^{gr}(b_1)^{n(j)}}.\] We know that 
$\nu_{p_j}(f)= (\nu_j\circ I(f))<0$ or in other words $f$ has a pole at the 
point $p_j$. \\ \\
It follows that \[i^{\rm gr} \left(\frac{(b_j)}{(b_1)^{n(j)}}\right)=
\frac{i(b_j)}{\sigma_1^{n(j)}}\] has a pole along 
$\pi_X^{-1}(p_j)$ which is only possible if $\sigma_1$ has a zero along the 
divisor $\pi_X^{-1}(p_j)$. But this can only be the case for 
a finite number of points $p_j$. \\ \\
This completes the proof of the Proposition \ref{criterion}.
\end{proof}
\section{Construction of the example.}
We will now construct our example of an approximable algebra $B=\oplus_n B_n$ 
such which satisfies the condition in Lemma \ref{criterion}.\\ \\
We will start by associating to any number $n$ an element of 
$\mathbb{N}^\mathbb{N}$. 
\begin{defn}
Let $n$ be a natural number. We denote by $I(n)$ the sequence
\[ I(n)= \left( n, \lfloor \frac{n}{2} \rfloor, \lfloor \frac{n}{4} \rfloor,
\lfloor \frac{n}{8} \rfloor, \ldots \right).\]
We further denote by $J(n)$ the sum of the elements of $I(n)$, ie.
\[ J(n)= n+ \lfloor \frac{n}{2} \rfloor+ \lfloor \frac{n}{4} \rfloor,
\lfloor \frac{n}{8} \rfloor + \ldots\]
\end{defn}
We will require the following superadditivity property of $I(n)$.
\begin{lemma}
For all integers $n$ and $m$ we have that \[I(n)+I(m)\leq I(n+m)\]
where we consider that $v\leq w$ if $v_i\leq w_i$ for every integer î
\end{lemma}
\begin{proof}
Immediate from the elementary fact that $\lfloor r+s\rfloor \geq \lfloor r \rfloor + \lfloor s \rfloor$ for any real numbers $r$ and $s$. \end{proof}
Note that in particular $J(n)+J(m)\leq J(n+m)$ for all integers $n$ and $m$.
\\ \\
We now choose an infinite sequence of distinct points in $\mathbb{C}$ which we 
denote by $z_0,\ldots, z_m\ldots$. Let $x$ be a variable and consider the
vector
\[ {\bf a}= \left((x-z_0), (x-z_1),\ldots \right).\]
We associate to any integer $n$ the polynomial in $x$
\[ P_n(x)= {\bf a}^{I(n)}= \prod_{i=0}^\infty 
(x-z_i)^{\lfloor \frac{n}{2^i}\rfloor}\]
and we are now in a position to define our algebra.
\begin{defn}
We define $B_n$ by $B_n\subset \mathbb{C}(x)$, 
\[ B_n= \left\{ \frac{Q(x)}{P_n(x)}| Q(x)\mbox{ polynomial of degree $J(n)$ }.\right\}.\]
The algebra $B$ is then the graded algebra $B=\oplus_nB_n$  
\end{defn}
We note that $\oplus_n B_n$ is an algebra because for all $n$ and $m$ we have 
that $I(n)+I(m)\leq I(n+m)$ so that $P_n\times P_m | P_{n+m}$.
\begin{rmk} Note that while each of the subsets $B_n$ is defined as a subset of 
$\mathbb{C}(x)$, the global algebra $\oplus_n B_n$  is not, since the subsets
$B_n, B_m\subset \mathbb{C}(x)$ are not typically disjoint.
\end{rmk}
We now show that  $B$ is approximable.
\begin{proposition}
The algebra $B=\oplus_n B_n$ is approximable.
\end{proposition}
\begin{proof}
The conditions (1) and (2) of approximability are immediate. We turn to
the proof of condition (3).\\ \\
Note that the image in $B_{pn}$ of ${\rm Sym}^n(B_p)$ is given by 
\[\frac{\mathbb{C}_{nJ(p)}[x]}{P_p(x)^n}\] which is of dimension $nJ(p)+1$ and 
that $B_{np}$ is itself  \[\frac {\mathbb{C}_{J(pn)}[x]}{P_{pn}}\] which is of dimension
$J(pn)+1$. We consider therefore the ratio 
\[ \frac{nJ(p)+1}{J(np)+1}\]
and we will show that for sufficiently large values of $p$ this ratio can be
made arbitrarily close to $1$.\footnote{Which is a slightly stronger condition than than required for approximability.} \\ \\
Note first that by the super-additivity property of $J$ we have that for all 
$p$ and $n$ \[nJ(p)\leq J(pn).\]
Moreover for all $n$ we have that $J(n)\leq 2n$.  
Set $m=\lfloor \log_2(p)\rfloor $. We have that 
\begin{equation*}
\begin{aligned} 
2p-J(p)& =&\sum_{k=0}^\infty \left( \frac{p}{2^k}-\lfloor \frac{p}{2^k}\rfloor\right)\\
       & =& \sum_{k=0}^m \left( \frac{p}{2^k}-\lfloor \frac{p}{2^k}\rfloor\right) +
\sum_{k=m+1}^\infty \left( \frac{p}{2^k}-\lfloor \frac{p}{2^k}\rfloor\right)\\
 &\leq & \sum_{k=1}^m  1 +\sum_{k=m+1}^\infty \frac{p}{2^{k+1}}\\
& = & m+ \frac{p}{2^m} \\
&\leq & \log_2(p)+2\frac{p}{\log_2(p)}.
\end{aligned}
\end{equation*}
In particular, $\lim_{p\rightarrow\infty} \frac{2p-J(p)}{p}=0.$ Suppose that $p_0$  such that for all $p\geq p_0$ we have that 
$2-\frac{J(p)}{p} \leq \epsilon$ so that for all $n$ and all $p\geq p_0$ 
\[ n(2-\epsilon)p\leq nJ(p)\leq J(pn)\leq 2pn\]
from which it follows that $ (1-\epsilon) \leq \frac{nJ(p)}{J(pn)}\leq 1$ 
and hence 
\[(1-\epsilon) \leq \frac{nJ(p)+1}{J(pn)+1}\leq 1\]
This completes the proof of the proposition.
\end{proof}
We will now identify $K^{\rm gr}(B)$
\begin{lemma}
Let $B$ be the algebra constructed above. Then $\mathbb{C}^{\rm gr}(B)=\mathbb{C}(x)$.
\end{lemma}
\begin{proof}
By definition, $K^{\rm gr}(B)=\left\{ \frac{f}{g}| f,g\in B_n\mbox{ for 
some } n \right\}$. There is a natural morphism of fields 
\[\phi:K^{\rm gr}(B)\rightarrow \mathbb{C}(x)\]
given by $\frac{ P(x)/ P_n(x)}{Q(x)/P_n(x)}\rightarrow P(x)/Q(x)$. Since 
$\phi\left(
\frac{x/ P_1(x)}{1/P_1(x)}\right)=x$, this map is surjective. As a morphism of 
fields, it is injective. This completes the proof of the lemma.
\end{proof}
However, if we set $b_1\in B_1=1$ then the set of valuations 
$\nu_{B, b_1}$ is infinite, since it contains $\nu_{z_i}$ for any choice of $i$.
\\ \\
But this now completes the proof of Theorem \ref{maintheorem}, since $B$ is approximable, but since the set $\nu_{B,b_1}$ is infinite, $B$ cannot be a graded sub-algebra of the space of sections of a big line bundle. 

\end{document}